\documentclass[10pt]{amsart}

\usepackage{amsfonts,amssymb,amscd,amstext}
\usepackage[a4paper,hmargin=4cm,tmargin=4.5cm, bmargin=4cm]{geometry}
\usepackage{hyperref}
\usepackage{graphicx}
\usepackage[utf8]{inputenc}
\usepackage{esint}
\usepackage{enumerate}

\renewcommand{\leq}{\leqslant}
\renewcommand{\geq}{\geqslant}
\renewcommand{\le}{\leqslant}
\renewcommand{\ge}{\geqslant}
\newcommand{\ptl}{\partial}

\newcommand{\rr}{{\mathbb{R}}}
\newcommand{\cc}{{\mathbb{C}}}

\newcommand{\la}{\lambda}

\newcommand{\hh}{{\mathbb{H}}}
\newcommand{\nn}{{\mathbb{N}}}
\newcommand{\sph}{{\mathbb{S}}}

\newcommand{\escpr}[1]{\big<#1\big>}
\newcommand{\Sg}{\Sigma}
\newcommand{\sg}{\sigma}
\newcommand{\Om}{\Omega}

\newcommand{\ga}{\gamma}

\newcommand{\nuh}{\nu_{h}}

\newcommand{\per}{P}

\newcommand{\cf}[1]{\mathbf{1}_{#1}}
\newcommand{\seq}[1]{\{#1_i\}_{i\in\nn}}

\DeclareMathOperator{\divv}{div}

\DeclareMathOperator{\intt}{int}

\DeclareMathOperator{\Isom}{Isom}

\newtheorem{theorem}{Theorem}[section]
\newtheorem{proposition}[theorem]{Proposition}
\newtheorem{lemma}[theorem]{Lemma}

\theoremstyle{definition}

\newtheorem{definition}{Definition} 

\theoremstyle{remark}

\newtheorem{remark}[theorem]{Remark}

\numberwithin{equation}{section}

\begin{document}

\title[Isoperimetric regions in contact sub-Riemannian manifolds]{Existence of isoperimetric regions in contact sub-Riemannian manifolds}

\author[M.~Galli]{Matteo Galli} \address{Departamento de
Geometr\'{\i}a y Topolog\'{\i}a \\
Universidad de Granada \\ E--18071 Granada \\ Espa\~na}
\email{galli@ugr.es}
\author[M.~Ritor\'e]{Manuel Ritor\'e} \address{Departamento de
Geometr\'{\i}a y Topolog\'{\i}a \\
Universidad de Granada \\ E--18071 Granada \\ Espa\~na}
\email{ritore@ugr.es}

\date{\today}

\thanks{Both authors supported by MEC-Feder grant MTM2007-61919}

\begin{abstract}
We prove existence of regions minimizing perimeter under a volume constraint in contact sub-Riemannian manifolds such that their quotient by the group of contact transformations preserving the sub-Riemannian metric is compact.
\end{abstract}

\maketitle

\thispagestyle{empty}

\bibliographystyle{amsplain}

\tableofcontents

\section{Introduction}
Isoperimetric inequalities are valuable tools in Analysis and Geometry. In a given space, an optimal isoperimetric inequality is provided by the isoperimetric profile function, i.e., the one that assigns to any volume $v>0$ the infimum of the perimeter of the sets of volume $v$. Isoperimetric regions are those for which this infimum is achieved. A relevant problem in this field is to analyze if isoperimetric regions exist in a given~space for any value of the volume, or equivalently if, for any fixed volume $v>0$, there~is a perimeter-minimizing set of volume $v$.

To consider this problem notions of volume and perimeter must be given. A~very general class where both can be defined is the one of metric measure spaces, widely studied in probability theory, where the volume is the measure and the perimeter is the classical Minkowski content, defined from the volume and the distance. A recently studied class is the one of Ahlfors-regular metric measure spaces supporting an $1$-Poincar\'e inequality \cite{Mi}, \cite{Am}, where functions of bounded variation and finite perimeter sets can be defined. Riemannian and sub-Riemannian manifolds are included in this class.

Isoperimetric inequalities have been considered in contact sub-Riemannian manifolds. Pansu \cite{MR676380} first proved an isoperimetric inequality of the type $|\ptl\Om|\ge C\,|\Om|^{4/3}$, for a given constant $C>0$, in the first Heisenberg group $\hh^1$. While the exponent $\tfrac{4}{3}$ is optimal, the constant $C$ is not. Pansu conjectured \cite{MR829003} that equality for the optimal constant is achieved by a distinguished family of spheres with constant mean curvature in $\hh^1$.  Chanillo and Yang \cite{MR2548248} recently extended Pansu's inequality to pseudo-hermitian $3$-manifolds without torsion. The interested reader may consult Chapter~8 of the monograph \cite{CDPT} for a detailed account on the isoperimetric inequality in the sub-Riemannian Heisenberg group $\hh^1$.

The problem of existence of isoperimetric regions has been widely considered in Riemannian manifolds. Classical compactness results of Geometric Measure Theory ensure existence in compact manifolds \cite{Gi}, \cite{Si}, \cite{Mo}. However, it is known that there exist complete non-compact Riemannian manifolds for which isoperimetric regions do not exist for any value of the volume, such as planes of revolution with strictly increasing Gauss curvature \cite[Thm.~2.16]{MR1883725}. On the other hand, isoperimetric regions exist for any given volume in complete surfaces with non-negative Gauss curvature \cite{MR1857855}. A very general existence result was stated by F.~Morgan in \cite{Mo} for Riemannnian manifolds which have compact quotient under the action of the isometry group. Its proof is modeled on a previous one of existence of clusters minimizing perimeter under  given volume constraints in Euclidean space \cite{Mo2}.

In sub-Riemannian Geometry, apart from the compact case, the only known existence result has been given by Leonardi and Rigot  for Carnot groups \cite{Le-Ri}. In their paper they made an extensive use of the properties of the isoperimetric profile in a Carnot group $\mathbb{G}$. Since isoperimetric regions in $\mathbb{G}$ are invariant by intrinsic dilations, the isoperimetric profile $I_\mathbb{G}$ of $\mathbb{G}$ is given by $I_\mathbb{G}(v)=Cv^q$, where $C$ is a positive constant and $q\in (0,1)$. In particular, the function $I_\mathbb{G}$ is strictly concave, a property that plays a fundamental role in their proof. Leonardi and Rigot also prove that isoperimetric sets are domains of isoperimetry in the particular case of the Heisenberg group. However, their result cannot be applied to some interesting sub-Riemannian groups, such as the roto-translational one \cite{CDPT}, which are not of Carnot type. Some of the crucial points of the proof of Leonardi and Rigot are discussed in \cite[\S~8.2]{CDPT}.

The aim of this paper is to prove in Theorem~\ref{th:main} an existence result for isoperimetric regions in contact sub-Riemannian manifolds such the quotient by the group of contact isometries, the diffeomorphisms that preserve the contact structure and the sub-Riemannian metric, is compact. This is the analog of Morgan's Riemannian result.

In the proof of Theorem~\ref{th:main} we follow closely Morgan's scheme: we pick a minimizing sequence of sets of volume $v$ whose perimeters approach the infimum of the perimeter of sets of volume $v$. If the sequence subconverges without losing any fraction of the original volume, the lower semicontinuity of the perimeter implies that the limit set is an isoperimetric region of volume $v$. If some fraction of the volume is missing then Proposition~\ref{prop:minseq} implies that the minimizing sequence can be broken into a converging part and a diverging one, the latter composed of sets of uniformly positive volume, see \cite{MR1883725}, \cite{MR1857855} and \cite{Ri-Ro2} for the Riemannian case. The converging part has a limit, which is an isoperimetric region for its volume, and is bounded by Lemma~\ref{lem:bounded}. Hence we can suitably translate the diverging part to recover some of the lost volume. An important point here is that we always recover a fixed fraction of the volume because of Lemma~\ref{lem:lr}, see \cite[Lemma~4.1]{Le-Ri}.

Along the proof of Theorem~\ref{th:main} two important technical points have to be solved, as mentioned in the previous paragraph. We prove in Lemma~\ref{lem:bounded} boundedness of the isoperimetric regions, and a structure result for minimizing sequences in Proposition~\ref{prop:minseq} . The key point to prove boundedness is the Deformation Lemma~\ref{addingperimeter}, where we slightly enlarge a given finite perimeter set producing a variation of perimeter which can be controlled by a multiple of the increase of volume. This is a extremely useful observation of Almgren \cite[V1.2(3)]{MR0420406}, \cite[Lemma~13.5]{Mo}. The Deformation Lemma is the only point where we strongly use the fact that our underlying sub-Riemannian manifold is of contact type, to construct a foliation by hypersurfaces with controlled mean curvature. Our proof of the Deformation Lemma~\ref{addingperimeter} does not seem to generalize easily to more general sub-Riemannian manifolds. The structure result for minimizing sequences appeared for the first time, although it was known to experts in Geometric Measure Theory, in \cite{MR1883725} for Riemannian surfaces, and in \cite{Ri-Ro2} for Riemannian manifolds of any dimension. In some cases, Proposition~\ref{prop:minseq} provides direct proofs of existence of isoperimetric regions.

We have organized this paper into five sections apart from this introduction. In Section~\ref{sec:preliminaries} we recall the necessary preliminaries about contact sub-Riemannian manifolds and metric measure spaces we shall use later. In Section~\ref{sec:rii} we obtain in Lemma~\ref{lem:relative} a relative isoperimetric inequality with uniform constant and radius in any compact set. This inequality is obtained from Jerison-Poincar\'e's inequality in Carnot-Carath\'eodory spaces \cite{Je}. It is then standard to prove Lemma~\ref{lem:isopsmall}, which yields a uniform isoperimetric inequality for small volumes, see also \cite{gromov-cc}. In Section~\ref{sec:deformation} we prove the crucial Deformation Lemma~\ref{addingperimeter} which allows us to deform a finite perimeter set modifying slightly its volume while keeping controlled the change of perimeter in terms of the variation of the volume. Lemma~\ref{addingperimeter} is proven by mapping, using Darboux's Theorem, the foliation by constant mean curvature Pansu's spheres in a punctured neighborhood of the origin in the Heisenberg group $\hh^n$, to our given contact sub-Riemannian manifold. We then prove that the mean curvature of the resulting foliation is bounded and we apply the sub-Riemannian Divergence Theorem to conclude the proof of the result. In Section~\ref{sec:structure} we prove a structure result for minimizing sequences in Proposition~\ref{prop:minseq}, and we state and prove some properties of the isoperimetric profile. Finally, in section~\ref{sec:main} we prove our main result, Theorem~\ref{th:main}, on existence of isoperimetric regions.


\section{Preliminaries}
\label{sec:preliminaries}

A \emph{contact manifold} \cite{Bl} is a $C^\infty$ manifold $M^{2n+1}$ of odd dimension so that there is an one-form $\omega$ such that $d\omega$ is non-degenerate when restricted to $\mathcal{H}:=\text{ker}(\omega)$. Since
\[
d\omega(X,Y)=X(\omega(Y))-Y(\omega(X))-\omega([X,Y]),
\]
the \emph{horizontal distribution} $\mathcal{H}:=\text{ker}(\omega)$ is completely non-integrable. One can easily prove the existence of a unique vector field $T$ in $M$ so that
\begin{equation}
\label{eq:reeb}
\omega(T)=1,\qquad (\mathcal{L}_T\omega)(X)=0,
\end{equation}
where $\mathcal{L}$ is the Lie derivative and $X$ is any smooth vector field on $M$. The vector field $T$ is usually called the \emph{Reeb vector field} of the contact manifold $M$. It is a direct consequence that $\omega\wedge (d\omega)^n$ is an orientation form in $M$.

A well-known example of a contact manifold is the Euclidean space $\rr^{2n+1}$ with the standard contact one-form
\begin{equation}
\label{eq:omega0}
\omega_0:=dt+\sum_{i=1}^n (x_idy_i-y_idx_i).
\end{equation}
A \emph{contact transformation} between contact manifolds is a diffeomorphism preserving the horizontal distributions. A \emph{strict contact transformation} is a diffeomorphism preserving the contact one-forms. A strict contact transformation preserves the Reeb vector fields. Darboux's Theorem \cite[Thm.~3.1]{Bl} shows that, given a contact manifold and some point $p\in M$, there is an open neighborhood $U$ of $p$ and a strict contact transformation $f$ from $U$ into an open set of $\rr^{2n+1}$ with its standard contact structure induced by $\omega_0$. Such a local chart will be called a \emph{Darboux chart}.

The length of a piecewise horizontal curve $\ga:I\to M$ is~defined by
\[
L(\ga):=\int_I |\ga'(t)|\,dt,
\]
where the modulus is computed with respect to the metric $g_\mathcal{H}$. The Carnot-Carath\'eo\-do\-ry distance $d(p,q)$ between $p$, $q\in M$ is defined as the infimum of the lengths of piecewise smooth horizontal curves joining $p$ and $q$. A minimizing geodesic is any curve $\ga:I\to M$ such that $d(\ga(t),\ga(t'))=|t-t'|$ for each $t$, $t'\in I$. We shall say that the sub-Riemannian manifold $(M,g_\mathcal{H},\omega)$ is complete if $(M,d)$ is a complete metric space. By Hopf-Rinow's Theorem \cite[p.~9]{gromov-ms} bounded closed sets are compact and each pair of points can be joined by a minimizing geodesic. From \cite[Chap.~5]{Mon} a minimizing geodesic in a contact sub-Riemannian manifold is a smooth curve that satisfies the geodesic equations, i.e., it is normal.

The metric $g_\mathcal{H}$ can be extended to a Riemannian metric $g$ on $M$ by requiring that $T$ be a unit vector orthogonal to $\mathcal{H}$. The scalar product of two vector fields $X$ and $Y$ with respect to the metric $g$ will be often denoted by $\escpr{X,Y}$. The Levi-Civita connection induced by $g$ will be denoted by $D$. An important property of the metric $g$ is that the integral curves of the Reeb vector field $T$ defined in \eqref{eq:reeb} are \emph{geodesics}, see \cite[Thm.~4.5]{Bl}. To check this property we observe that condition $(\mathcal{L}_T\omega)(X)$ in \eqref{eq:reeb} applied to a horizontal vector field $X$ yields $\omega([T,X])=0$ so that $[T,X]$ is horizontal. Hence, for any horizontal vector field $X$, we have
\[
\escpr{X,D_TT}=-\escpr{D_TX,T}=-\escpr{D_XT,T}=0,
\] 
where in the last equality we have used $|T|=1$. Since we trivially have $\escpr{T,D_TT}=0$, we get $D_TT=0$, as we claimed.

A usual class defined in contact geometry is the one of contact Riemannian manifolds, see \cite{Bl}, \cite{MR1000553}. Given a contact manifold, one can assure the existence of a Riemannian metric $g$ and an $(1,1)$-tensor field $J$ so that
\begin{equation}
\label{eq:contactriem}
g(T,X)=\omega(X), \quad 2g(X,J(Y))=d\omega(X,Y),\quad J^2(X)=-X+\omega(X)\,T.
\end{equation}
The structure given by $(M,\omega,g,J)$ is called a contact Riemannian manifold. The class of contact sub-Riemannian manifolds is different from this one. Recall that, in our definition, the metric $g_\mathcal{H}$ is given, and it is extended to a Riemannian metric $g$ in $TM$. However, there is not in general an $(1,1)$-tensor field $J$ satisfying all conditions in \eqref{eq:contactriem}. Observe that the second condition in \eqref{eq:contactriem} uniquely defines $J$ on $\mathcal{H}$, but this $J$ does not satisfy in general the third condition in \eqref{eq:contactriem}, as it is easily seen in $(\rr^3,\omega_0)$ choosing an appropriate  positive definite metric in $\text{ker}(\omega_0)$.

The Riemannian volume form $dv_g$ in $(M,g)$ coincides with Popp's measure \cite[\S~10.6]{Mon}. The volume of a set $E\subset M$ with respect to the Riemannian metric $g$ will be denoted by $|E|$. 

A \emph{contact isometry} in $(M,g_\mathcal{H},\omega)$ is a strict contact transformation that preserves $g_\mathcal{H}$. Contact isometries preserve the Reeb vector fields and they are isometries of the Riemannian manifold $(M,g)$. The group of contact isometries of $(M,g_\mathcal{H},\omega)$ will be denoted by $\Isom_{\omega}(M,g)$.

It follows from \cite[Thm.~1]{NSW} that, given a compact set $K\subset M$ there are positive constants $\ell$, $L$, $r_0$, such that $M$ is \emph{Ahlfors-regular}
\begin{equation}
\label{eq:ahlfors}
\ell r^Q\leq |B(x,r)|\leq L r^Q,
\end{equation}
for all $x\in K$, $0<r<r_0$. Here $Q$ is the \emph{homogeneous dimension} of $M$, defined as
\begin{equation}
\label{eq_homdimension}
Q:=2n+2.
\end{equation}
Related to the homogeneous dimension we shall also consider the isoperimetric exponent
\begin{equation}
\label{eq:defq}
q:=(Q-1)/Q.
\end{equation}
In the case of contact sub-Riemannian manifolds this result also follows taking Darboux charts. Inequalities \eqref{eq:ahlfors} immediately imply the \emph{doubling property}: given a compact set $K\subset M$, there are positive constants $C$, $r_0$ such that
\begin{equation}
\label{eq:doubling}
|B(x,2r)|\le C |B(x,r)|,
\end{equation}
for all $x\in K$, $0<r<r_0$. Moreover, \eqref{eq:ahlfors} also implies that, given a compact subset $K\subset M$, there are positive constants $C$, $r_0$, such that
\begin{equation}
\label{eq:homogeneity}
\frac{|B(x_0,r)|}{|B(x,s)|}\geq C\bigg(\frac{r}{s}\bigg)^Q,
\end{equation}
for any $x_0\in K$, $x\in B(x_0,r)$, $0<r\le s<r_0$.

Given a Borel set $E\subset M$ and an open set $\Om\subset M$, the \emph{perimeter} of $E$ in $\Omega$ can be~defined, following the Euclidean definition by De Giorgi, by
\begin{equation*}
\per(E,\Omega):=\sup\bigg\{\int_{E}\divv X\, dv_g: X\in \frak{X}_0^1(\Om), X\ \text{horizontal}, |X|\le 1 \bigg\},
\end{equation*}
where $\frak{X}_0^1(\Om)$ is the space of vector fields of class $C^1$ and compact support in $\Om$ and $\divv$ is the divergence in the Riemannian manifold $(M,g)$. When $\Omega=M$ we define $P(E):=\per(E,M)$. A set $E$ is called of \emph{finite perimeter} if $\per(E)<+\infty$, and of \emph{locally finite perimeter} if $P(E,\Om)<+\infty$ for any bounded open subset $\Om\subset M$. See \cite{FSSC4} and \cite{Ga-Nh} for similar definitions.

A function $u\in L^1(M)$ is of \emph{bounded variation} in an open set $\Omega$ if 
\begin{equation*}
|Du|(\Om):=\sup \bigg\{ \int_{\Omega}u \divv X \,dv_g : X\in C_0^1(M), X\ \text{horizontal}, |X|\le 1,\,\text{supp}\,X\subset\Om\bigg\}
\end{equation*}
is finite. We shall say that $|Du|(\Om)$ is the \emph{total variation} of $u$ in $\Om$. The space of functions with bounded variation in $M$ will be denoted by $BV(M)$. If $u$ is a smooth function then
\[
|Du|(\Om)=\int_\Om |\nabla_hu|\,dv_g,
\]
where $\nabla_hu$ is the orthogonal projection to $\mathcal{H}$ of the gradient $\nabla u$ of $u$ in $(M,g)$.

It follows easily that $P(E,\Om)$ is the total variation of the characteristic function $\cf{E}$ of $E$. A sequence of finite perimeter sets $\{E_i\}_{i\in\nn}$ converges to a finite perimeter set $E$ if $\cf{E_i}$ converges to $\cf{E}$ in $L^1_{loc}(M)$.

Finite perimeter sets are defined up to a set of measure zero. We can always choose a representative so that all density one points are included in the set and all density zero points are excluded \cite[Chap.~3]{Gi}. We shall always take such a representative without an explicit mention.

There is a more general definition of functions of bounded variation and of sets of finite perimeter in metric measure spaces, using a relaxation procedure, using as energy functional the $L^1$ norm of the minimal upper gradient, \cite{Mi}, \cite{Am}. If $(M,g_\mathcal{H},\omega)$ is a contact sub-Riemannian manifold then the definition of perimeter given above coincides with the one in \cite{Mi}, \cite{Am}. See \cite[\S~5.3]{Mi}, \cite[Ex.~3.2]{Am}.

In case $E$ has $C^1$ boundary $\Sg$, it follows from the Divergence Theorem in the~Riemannian manifold $(M,g)$ that the perimeter $P(E)$ coincides with the sub-Riemannian area of $\Sg$ defined by
\begin{equation}
\label{eq:area}
A(\Sg):=\int_{\Sg} |N_h|\,d\Sg,
\end{equation}
where $N$ is a unit vector field normal to $\Sg$, $N_h$ the orthogonal projection of $N$ to the horizontal distribution, and $d\Sg$ is the Riemannian measure of $\Sg$.

 The following usual properties for finite perimeter sets $E$, $F\subset M$ in an open set $\Om\subset M$ are proven in \cite{Mi}
\begin{enumerate}
\item $\per(E,\Om)=\per(F,\Om)$ when the symmetric difference $E\triangle F$ satisfies $|(E\triangle F)\cap\Om)|=0$.
\item $\per(E\cup F,\Om)+\per(E\cap F,\Om)\le\per(E,\Om)+\per(F,\Om)$.
\item $\per(E,\Om)=\per(M\setminus E,\Om)$.
\end{enumerate}
The set function $\Om\mapsto P(E,\Om)$ is the restriction to the open subsets of the finite Borel measure $P(E,\cdot\,)$ defined by
\begin{equation}
\label{eq:permeasure}
P(E,B):=\inf\{P(E,A): B\subset A, A\ \text{open}\},
\end{equation}
where $B$ is a any Borel set.

We fix a point $p\in M$ and we consider the open balls $B_r:=B(p,r)$, $r>0$. Then the following property holds from the definitions
\begin{equation}
\label{eq:perinter}
\per(E\cap B_r)\le\per(E,B_r)+\per(E\setminus B_r,\ptl B_r),
\end{equation}
where $P(E\setminus B_r,\ptl B_r)$ is defined from \eqref{eq:permeasure}. 

The following results are proved in general metric measure spaces

\begin{proposition}[Lower semicontinuity \cite{Am},\cite{Mi}] The function $E\rightarrow \per(E,\Omega)$ is lower semicontinuous with respect to the $L^1 (\Om)$ topology.
\end{proposition}

\begin{proposition}[Compactness \cite{Mi}] 
\label{prop:compactness}
Let $\seq{E}$ be a sequence of finite perimeter sets such that $\{\cf{E_i}\}_{i\in\nn}$ is bounded in $L^1_{loc}(M)$ norm and satisfying $\sup_i\per(E_i,\Omega)<+\infty$ for any relatively compact open set $\Om\subset M$. Then there exists a finite perimeter set $E$ in $M$ and a subsequence $\{\cf{E_{n_i}}\}_{i\in\nn}$ converging to $\cf{E}$ in $L^1_{loc}(M)$.
\end{proposition}


\begin{theorem}[Gauss-Green for finite perimeter sets]
Let $E\subset M$ be a set of finite perimeter. Then there exists a $\per(E)$-measurable vector field $\nu_E\in TM$ such that 
\begin{equation*}
-\int_{E}\divv X\,dv_g=\int_{M}g_{\mathcal{H}}(\nu_E,X)\,d\per(E),
\end{equation*}
for all $X\in \mathcal{H}$ and $|\nu_E|=1$ for $\per(E)$-a.e. $x\in M$.
\end{theorem}

The proof consists essentially in taking local coordinates and applying Riesz Representation Theorem \cite[\S~1.8]{EG} to the linear functional $f\mapsto -\int f\,\divv_{\mathcal{H}}Xdv_g$, where $f$ is any function with compact support in $M$. This result was proven in the Heisenberg group $\hh^n$ in \cite{FSSC4}.

\begin{definition} Let $E$ be a finite perimeter set. The \emph{reduced boundary} $\partial^* E$ is composed of the points $x\in\partial E$ which satisfy 
\begin{itemize}
\item [(i)] $\per(E,B_r (x))>0$, for all $r>0$;
\item[(ii)] exists $\lim\limits_{r\rightarrow 0}\fint\limits_{B_r(x)}\nu_E d\per(E)$ and its modulus is one.
\end{itemize}
\end{definition}

The following approximation result, whose proof is a straightforward adaptation of the Euclidean one, \cite[Chap.~1]{Gi}, holds.

\begin{proposition} Let $(M,g_\mathcal{H},\omega)$ be a contact sub-Riemannian manifold, and let $u\in BV(\Omega)$. Then there exists a sequence $\seq{u}$ of smooth functions such that $u_i\rightarrow u$ in $L^1(\Omega)$ and $\lim_{i\rightarrow +\infty} |\nabla_h u_i|(\Omega)=|\nabla_h u|(\Omega)$.
\end{proposition}

The localization lemma \cite[Lemma~3.5]{Am}, see also \cite{Mi}, allows us to prove

\begin{proposition}
Let $(M,g_\mathcal{H},\omega)$ be a contact sub-Riemannian manifold, $E\subset M$ a finite perimeter set, $p\in M$, and $B_r:=B(p,r)$. Then, for almost all $r>0$, the set $E\setminus B_r$ has finite perimeter, and
\[
\per(E\setminus B_r,\ptl B_r)\le\frac{d}{dr}\,|E\cap B_r|.
\]
\end{proposition}

%

The \emph{isoperimetric profile} of $M$ is the function $I_M:(0,|M|)\rightarrow \rr^+\cup \{0\}$ given by
\begin{equation*}
I_M(v):=\inf \{\per(E):E\subset M, |E|=v\}.
\end{equation*}
A set $E\subset M$ is an \emph{isoperimetric region} if $P(E)=I_M(|E|)$. The isoperimetric profile must be seen as an optimal isoperimetric inequality in the manifold $M$, since for any set $E\subset M$ we have
\[
P(E)\ge I_M(|E|),
\]
with equality if and only if $E$ is an isoperimetric region.

\section{A relative isoperimetric inequality and an isoperimetric inequality for small volumes}
\label{sec:rii}

In this section we consider a contact sub-Riemannian manifold $(M,g_\mathcal{H},\omega)$. We shall say that $M$ supports a \emph{$1$-Poincar\'e inequality} if there are constants $C_P$, $r_0>0$ such that
\[
\int_{B(p,r)}|u-u_{p,r}|\,dv_g\le C_pr\int_{B(p,r)} |\nabla_hu|\,dv_g
\]
holds for every $p\in M$, $0<r<r_0$, and $u\in C^\infty(M)$. Here $u_{p,r}$ is the average value of the function $u$ in the ball $B(p,r)$ with respect to the measure $dv_g$
\[
u_{p,r}:=\frac{1}{|B(p,r)|}\,\int_{B(p,r)} u\,dv_g,
\]
that will also be denoted by
\[
\fint_{B(x,r)} u\,dv_g
\]
We shall prove that a $1$-Poincar\'e inequality holds in $M$ provided $M/\Isom_\omega(M,g)$ is compact, using the following result by Jerison

\begin{theorem}[{\cite[Thm.~2.1]{Je}}]
Let $X_1$, $\ldots$, $X_m$ be $C^\infty$ vector fields satisfying H\"ormander's condition defined on a neighborhood $\Om$ of the closure $\overline{E}_1$ of the Euclidean unit ball $E_1\subset\rr^d$.

Then there exists constants $C>0$, $r_0>0$ such that, for any $x\in E_1$ and every $0<r<r_0$ such that $B(x,2r)\subset \Om$,
\begin{equation}
\label{eq:jerison}
\int_{B(x,r)} |f-\tilde{f}_{x,r}|\,d\mathcal{L}\leq Cr\int_{B(x,r)} \bigg(\sum_{i=1}^m X_i(f)^2\bigg)^{1/2}\,d\mathcal{L},
\end{equation}
for any $f\in C^\infty(\overline{B}(x,r))$, where the integration is taken with respect to the Lebesgue measure $\mathcal{L}$, the balls are computed with respect to the Carnot-Carath\'eodory distance associated to $X_1$, $\ldots$, $X_m$, and $\tilde{f}_{x,r}$ is the mean with respect to Lebesgue measure.
\end{theorem}

\begin{remark}
Jerison really proved the $2$-Poincar\'e inequality
\[
\int_{B(x,r)} |f-\tilde{f}_{x,r}|^2\,d\mathcal{L}\leq Cr^2\int_{B(x,r)} \bigg(\sum_{i=1}^m X_i(f)^2\bigg)\,d\mathcal{L}.
\]
However, as stated by Hajłasz and Koskela \cite[Thm.~11.20]{HK}, his proof also works for the $L^1$ norm in both sides of the inequality.
\end{remark}

\begin{remark}
The dependence of the constants $C$, $r_0$ is described in \cite[p.~505]{Je}.
\end{remark}

Using Jerison's result we can easily prove

\begin{lemma}[Poincar\'e's inequality]
\label{lem:poincare}
Let $(M,g_\mathcal{H},\omega)$ be a contact sub-Riemannian manifold, and $K\subset M$ a compact subset. Then there exist constants $C_P$, $r_0>0$, only depending on $K$, such that
\begin{equation}
\label{eq:poincare}
\int_{B(p,r)}|u-u_{p,r}|\,dv_g\le C_pr\int_{B(p,r)} |\nabla_hu|\,dv_g,
\end{equation}
for all $p\in K$, $0<r<r_0$, $u\in C^\infty(M)$.
\end{lemma}

\begin{proof}
For every $p\in K$ we consider a Darboux chart centered at $p$, i.e., an open neighborhood $U_p$ of $p$ together with a diffeomorphism $\phi_p:U_p\to\rr^{2n+1}$ with $\phi_p(p)=0$ and $\phi_p^*\,\omega_0=\omega$, where $\omega_0$ is the standard contact form in Euclidean space.

We denote by $h_\la:\rr^{2n+1}\to\rr^{2n+1}$, for $\la>0$, the intrinsic dilation of ratio $\la$, defined by $h_\la(z,t):=(\la z,\la^2 t)$, for $(z,t)\in\cc^n\times\rr\equiv\rr^{2n+1}$. For every $p\in K$ we choose $\la(p)>0$ so that the image of $U_p$ by $\varphi_p:=h_{\la(p)}\circ\phi_p$ contains the closure $\overline{E}_1$ of the unit ball $E_1\subset\rr^{2n+1}$. From the open covering $\{\varphi_p^{-1}(E_1)\}_{p\in K}$ of $K$ we extract a finite subcovering $\varphi_{p_1}^{-1}(E_1)$, $\ldots$, $\varphi_{p_r}^{-1}(E_1)$. From now on we fix some $p_i$, $i=1,\ldots r$, and we take $p=p_i$, $\varphi=\varphi_i$, $U=U_{p_i}$.

We consider the scalar product $h:=(\varphi^{-1})^*\,g_\mathcal{H}$ in the contact distribution $\mathcal{H}_0:=\text{ker}(\omega_0)$. Let $X_1,\ldots, X_{2n}$ be an orthonormal basis of $\mathcal{H}_0$ with respect to $h$ in $\Om$. Observe that $\varphi$ is a contact transformation from $(U,\omega)$ to $(\Om,\omega_0)$ that preserves the sub-Riemannian metrics. Hence $\varphi$ is an isometry between metric spaces when we consider on $(U,\omega)$ its associated Carnot-Carath\'eodory distance $d$ and on $(\Om,\omega_0)$ the distance induced by the family of vector fields $X_1,\ldots,X_{2n}$. Moreover, if $u\in C^\infty(M)$ then, for every $p\in U$, we have
\begin{equation}
\label{eq:nablax}
|(\nabla_h u)_p|=\bigg(\sum_{i=1}^{2n} (X_i)_p(\varphi\circ u)\bigg)^{1/2}.
\end{equation}
Let $\mu:=\varphi^{-1}(dv_g)$. Since $\varphi$ is a diffeomorphism, $\mu$ and $d\mathcal{L}$ satisfy
\begin{equation}
\label{eq:ell}
\ell\,\mathcal{L}(E)\leq \mu(E)\leq L\,\mathcal{L}(E),
\end{equation}
for some constants $\ell$, $L>0$, and $E$ contained in a compact neighborhood of $E_1$ in $\varphi(U)$.

By Jerison's result, there are $C$, $r_0>0$ so that
\[
\int_{B(x,r)}|f-\tilde{f}_{x,r}|\,d\mathcal{L}\leq Cr\int_{B(x,r)}\bigg(\sum_{i=1}^{2n} (X_i(f))^2\bigg)^{1/2}d\mathcal{L}
\]
for all $f\in C^\infty(\Om)$, $z\in E_1$. Since
\begin{equation*}
\int_{B(x,r)}|f-\tilde{f}_{x,r}|\,d\mathcal{L}=\frac{1}{\mathcal{L}(B(x,r))}\,\int_{B(x,r)}\int_{B(x,r)}|f(y)-f(z)|\,d\mathcal{L}(y)\,d\mathcal{L}(z),
\end{equation*}
we can use inequalities \eqref{eq:ell} to prove that there is $C'>0$ such that
\[
\int_{B(x,r)}|f-\tilde{f}_{x,r}|\,d\mathcal{L}\geq C' \int_{B(x,r)}|f-{f}_{x,r}|\,d\mu,
\]
where $f_{x,r}$ is the mean of $f$ in the ball $B(x,r)$. So we obtain from \eqref{eq:jerison} and again from \eqref{eq:ell} that there are $C$, $r_0>0$ so that
\[
\int_{B(x,r)}|f-f_{x,r}|\,d\mu\le Cr\int_{B(x,r)}\bigg(\sum_{i=1}^{2n}X_i(f)^2\bigg)^{1/2}d\mu.
\]

From the definition of $\mu$, the fact that $\varphi$ is an isometry, and \eqref{eq:nablax} we obtain \eqref{eq:poincare} for $p\in\varphi^{-1}(E_1)$. We repeat this process for every open set $\varphi_{p_i}^{-1}(E_1)$, $i=1,\ldots,r$. Taking the maximum of the constants $C$ so obtained and the minimum of the radii $r_0$ it follows that \eqref{eq:poincare} holds for all $p\in K$.
\end{proof}

Using this result we get 

\begin{lemma}
\label{lem:poincareisom}
Let $(M,g_\mathcal{H},\omega)$ be a contact sub-Riemannian manifold such that the~quotient $M/\Isom_\omega(M,g)$ is compact. Then there exist constants $C_P$, $r_0>0$, only depending on $M$, such that
\begin{equation}
\label{eq:poincareisom}
\int_{B(p,r)}|u-u_{p,r}|\,dv_g\le C_Pr\int_{B(p,r)} |\nabla_hu|\,dv_g,
\end{equation}
for all $p\in M$, $0<r<r_0$, $u\in C^\infty(M)$.
\end{lemma}

\begin{remark}
Poincar\'e's inequality also holds for functions of bounded variation by an approximation argument, see \cite{Gi}.
\end{remark}

From the $1$-Poincar\'e inequality \eqref{eq:poincare} and inequality \eqref{eq:homogeneity} we can prove, using Theorem~5.1 and Corollary~9.8 in \cite{HK} (see also Remark~3 after the statement of Theorem~5.1 in \cite{HK}), that, given a compact set $K\subset M$, there are positive constants $C$, $r_0$ so that
\begin{equation}
\label{eq:q-poincare}
\bigg(\fint_{B(x,r)}|u-u_{x,r}|^{Q/(Q-1)}\bigg)^{(Q-1)/Q}\le Cr\bigg(\fint_{B(x,r)}|\nabla_hu|\bigg),
\end{equation}
for all $u\in C^\infty(M)$, $x\in K$, $0<r<r_0$. Furthermore, it is well-known that the $q$-Poincar\'e's inequality \eqref{eq:q-poincare} implies the following relative isoperimetric inequality, \cite{EG} and \cite{Gi}

\begin{lemma}[Relative isoperimetric inequality] 
\label{lem:relative}
Let $(M,g_\mathcal{H},\omega)$ be a contact sub-Riemannian manifold, and $K\subset M$ a compact subset. There exists constants $C_I>0$, $r_0>0$, only depending on $K$, so that, for any set $E\subset M$ with locally finite perimeter, we have
\begin{equation}
\label{relativeisop}
C_I\min\big\{|\,E\cap B(x,r)\,|,|\,E^c \cap B(x,r)\,|\big\}^{(Q-1)/Q}\leq \per(E,B(x, r)),
\end{equation}
for any $x\in K$.  
\end{lemma}

\begin{remark} 
A relative isoperimetric inequality in compact subsets of $\rr^n$ for sets $E$ with $\mathcal{C}^1$ boundary was proven in \cite{CDG} for the sub-Riemannian structure given by a family of H\"ormander vector fields. As the authors remark their result holds for any family of vector fields on a connected manifold.
\end{remark}

\begin{remark}
As for Poincar\'e's inequality, the relative isoperimetric inequality \eqref{relativeisop} holds in the whole of $M$ provided $M/\Isom_\omega(M,g)$ is compact.
\end{remark}

\begin{lemma}[Isoperimetric inequality for small volumes]
\label{lem:isopsmall}
Let $(M,g_{\mathcal{H}},\omega)$ be a contact sub-Riemannian manifold so that the quotient $M/\Isom_\omega(M)$ is compact. Then there exists $v_0>0$ and $C_I>0$ such that
\begin{equation}
\label{eq:relisop}
\per(E)\ge C_I|E|^{(Q-1)/Q},
\end{equation}
 for any finite perimeter set $E\subset M$ with $|E|<v_0$.
\end{lemma}

\begin{proof}
This is a classical argument \cite[Lemma~4.1]{Le-Ri}. We fix $\delta>0$ small enough so that~Poin\-car\'e's inequality holds for balls of radius smaller than or equal to $\delta$. Since $M/\Isom_\omega(M)$ is compact, there exists $v_0>0$ so that $|B(x,\delta)|\geq 2v_0$ holds for all $x\in M$. Let $E\subset M$ be a set of finite perimeter with $|E|<v_0$. We fix a maximal family of points $\{x_i\}_{i\in\nn}$ with the properties
\begin{equation}
\label{eq:xi}
d(x_i,x_j)\geq \frac{\delta}{2}\ \ \text{for}\ i\neq j, \qquad E\subseteq\bigcup\limits_{i\in\nn}B(x_i,\delta).
\end{equation}
Letting $q:=(Q-1)/Q$ we have
\begin{equation}
\label{eq:estimateEq}
|E|^q\leq \left(\sum\limits_{i\in\nn} |B(x_i,\delta)\cap E|\right)^q\leq \sum\limits_{i\in\nn}|B(x_i,\delta)\cap E|^q\leq
C_1 \sum\limits_{i\in\nn}\per(E, B(x_i, \delta))
\end{equation}
from \eqref{eq:xi}, the concavity of the function $x\mapsto x^q$, and the relative isoperimetric inequality in Lemma~\ref{lem:relative}. For $z\in M$, we~define $A(z):=\{x_i: z\in B(x_i,\delta)\}$, so that $B(x,\delta/4)\subset B(z,2\delta)$ and $B(z,\delta/4)\subset B(x,2\delta)$ for $x\in A(z)$. Since the balls $B(x_i,\delta/4)$ are disjoint by \eqref{eq:xi}, we get
\begin{equation}
\label{eq:card1}
\#A(z) \min\limits_{x\in A(z)} |B(x,\delta/4)|\le\bigg|\bigcup\limits_{x\in A(z)} B(x,\delta/4)\bigg|\leq |B(z,2\delta)|.
\end{equation}
On the other hand, since $B(z,\delta/4)\subset B(x,2\delta)$ we have
\begin{equation}
\label{eq:card2}
|B(x,\delta/4)|\ge C_D^{-3}\,|B(z,\delta/4)|,
\end{equation}
where $C_D>0$ is the doubling constant. We conclude from \eqref{eq:card1} and \eqref{eq:card2} that
\[
\# A(z)\le C^6_D,
\]
and so
\[
\sum_{i\in\nn} \per(E,B(x_i,\delta))\le C\,\per(E).
\]
This inequality and \eqref{eq:estimateEq} yields \eqref{eq:relisop}.
\end{proof}

\begin{remark}
Another approach to isoperimetric inequalities in Carnot-Carath\'eodory spaces is provided by Gromov \cite[\S~2.3]{gromov-cc}.
\end{remark}

\begin{remark}
An isoperimetric inequality for small volumes in compact Riemannian manifolds was proven by Berard and Meyer \cite{MR690651}.
\end{remark}

\section{The Deformation Lemma. Boundedness of isoperimetric regions}
\label{sec:deformation}

In order to prove Theorem~\ref{th:main}, we need to construct a foliation of a punctured~neighborhood of any point in $M$ by smooth hypersurfaces with bounded mean curvature. We briefly recall this definition. Let $\Sg\subset M$ be a $C^2$ hypersurface in $M$. The \emph{singular set} $\Sg_0$ of $\Sg$ is the set of points in $\Sg$ where the tangent hyperplane to $\Sg$ coincides with the horizontal distribution. If $\Sg$ is orientable then there exists a globally defined unit normal vector field $N$ to $\Sg$ in $(M,g)$, from which a horizontal unit normal $\nu_h$ can be defined on $\Sg\setminus\Sg_0$ by
\begin{equation}
\label{eq:nuh}
\nuh:=\frac{N_h}{|N_h|},
\end{equation}
where $N_h$ is the orthogonal projection of $N$ to the horizontal distribution. The sub-Riemannian \emph{mean curvature} of $\Sg$ is the function, defined in $\Sg\setminus\Sg_0$, by
\begin{equation}
\label{eq:h}
H:=-\sum_{i=1}^{2n-1}\escpr{D_{e_i}\nuh,e_i},
\end{equation}
where $D$ is the Levi-Civita connection in $(M,g)$, and $\{e_1,\ldots,e_{2n-1}\}$ is an orthonormal basis of  $T\Sg\cap\mathcal{H}$. We recall that, given a vector field $X$ defined on $\Sg$, the divergence of $X$ in $\Sg$, $\divv_\Sg X$, is defined by
\begin{equation}
\label{eq:divsg}
\divv_\Sg X:=\sum_{i=1}^{2n}\escpr{D_{e_i}\nuh,e_i},
\end{equation}
where $\{e_1,...,e_{2n}\}$ is an orthonormal basis of $T\Sg$. 

We define the tensor
\begin{equation}
\label{eq:sigma}
\sg(X,Y):=\escpr{D_XT,Y},
\end{equation}
where $X$ and $Y$ are vector fields on $M$. In the case of the Heisenberg group we have $D_XT=J(X)$, so that $\sg(X,Y)=\escpr{J(X),Y}$. 

At every point of $\Sg\setminus\Sg_0$, we may choose an orthonormal basis of $T\Sg$ consisting on an orthonormal basis $\{e_1,\ldots,e_{2n-1}\}$ of $T\Sg\cap\mathcal{H}$ together with the vector
\begin{equation}
\label{eq:S}
S:=\escpr{N,T}\,\nuh-|N_h|\,T,
\end{equation}
which is orthogonal to $N$ and of modulus $1$. Hence we obtain in $\Sg\setminus\Sg_0$
\begin{equation}
\label{eq:hdivsg}
\divv_\Sg\nuh=-H+\escpr{D_S\nuh,S}.
\end{equation}
From \eqref{eq:S} and equality $|\nuh|=1$ we immediately get $\escpr{D_S\nuh,S}=-|N_h|\,\escpr{D_S\nuh,T}$, which is equal to $|N_h|\,\sg(\nuh,S)$. Since the vector field $S$ can be rewritten in the form $S=|N_h|^{-1}\big(\escpr{N,T}\,N-T\big)$, and $D_TT=0$, we finally get
\begin{equation*}
\escpr{D_S\nuh,S}=\escpr{N,T}\,\sg(\nuh,N),
\end{equation*}
and so
\begin{equation}
\label{eq:hdivsgfinal}
\divv_\Sg \nuh=-H+\escpr{N,T}\,\sg(\nuh,N)
\end{equation}

The mean curvature \eqref{eq:h} appears in the expression of the first derivative of the sub-Riemannian area functional \eqref{eq:area}.

\begin{lemma}
Let $\Sg\subset M$ be an orientable hypersurface of class $C^2$ in a contact sub-Riemannian manifold $(M,g_\mathcal{H},\omega)$, and let $U$ be a vector field with compact support in $M\setminus\Sg_0$ and associated one-parameter family of diffeomorphisms $\{\varphi_s\}_{s\in\rr}$. Then
\begin{equation}
\label{eq:1stvar}
\frac{d}{ds}\bigg|_{s=0} A(\varphi_s(\Sg))=-\int_\Sg H\,\escpr{U,N}\,d\Sg.
\end{equation}
\end{lemma}

\begin{proof}
Let $u:=\escpr{U,N}$. Following the proof of \cite[Lemma~3.2]{Ri-Ro3} we obtain
\[
\frac{d}{ds}\bigg|_{s=0} A(\varphi_s(\Sg))=\int_\Sg \{U^{\perp}(|N_h|)+|N_h| \divv_{\Sg}U^{\perp}\}\,d\Sg.
\]
For the first summand in the integrand we obtain
\begin{align*}
U^{\perp}(|N_h|)=U^\perp\big(\escpr{N,\nuh}\big)&=\escpr{D_{U^{\perp}} N,\nu_h}+\escpr{N,D_{U^\perp}\nuh}
\\
&=-\escpr{\nabla_\Sg u,\nuh}-\escpr{N,T}\,\sg(\nuh,U^\bot)
\\
&=-(\nuh)^\top(u)-\escpr{N,T}\,\sg(\nuh,U^\bot),
\end{align*}
since $D_{U^{\perp}}N=-\nabla_{\Sg} u$. So we get from the previous formula
\begin{equation*}
\begin{aligned}
U^\bot(|N_h|)&+|N_h|\divv_{\Sg}U^\bot=\\
&=-(\nu_h)^{\top}(u)-\escpr{N,T}\,\sg(\nuh,U^\bot)+u\divv_\Sg (|N_h| N)\\
&=-\divv_\Sg (u(\nu_h)^{\top})+u\divv_\Sg (\nu_h)^{\top}-\escpr{N,T}\,\sg(\nuh,U^\bot)+u \divv_\Sg (|N_h|N)\\
&=-\divv_\Sg (u(\nu_h)^{\top})+u\divv_\Sg (\nu_h)-u\,\escpr{N,T}\,\sg(\nuh,N),
\end{aligned}
\end{equation*}
where we have used $\nu_h=(\nu_h)^{\top}+|N_h|N$ in the final step. Since $U$ has compact support out of $\Sg_0$, where $\nuh$ is well defined, we conclude from the Riemannian Divergence Theorem and \eqref{eq:hdivsgfinal}
\begin{equation*}
\frac{d}{ds}\bigg|_{s=0} A(\varphi_s(\Sg))=\int_\Sg u\,\{\divv_\Sg (\nu_h)-\escpr{N,T}\,\sg(\nuh,N)\}\,d\Sg=-\int_\Sg H\,\escpr{U,N}\,d\Sg,
\end{equation*}
which completes the proof of the Lemma.
\end{proof}

The local model of a sub-Riemannian manifold is the contact manifold $(\rr^{2n+1},\omega_0)$, where $\omega_0:=dt+\sum_{i=1}^n (x_idy_i-y_idx_i)$ is the standard contact form in $\rr^{2n+1}$, together with an arbitrary positive definite metric $g_{\mathcal{H}_0}$ in $\mathcal{H}_0$. A basis of the horizontal distribution is given by
\[
X_i:=\frac{\ptl}{\ptl x_i}+y_i\,\frac{\ptl}{\ptl t},\qquad Y_i:=\frac{\ptl}{\ptl y_i}+x_i\,\frac{\ptl}{\ptl t},\qquad i=1,\ldots,n,
\]
and the Reeb vector field is
\[
T:=\frac{\ptl}{\ptl t}.
\]
The metric $g_{\mathcal{H}_0}$ will be extended to a Riemannian metric on $\rr^{2n+1}$ so that the Reeb vector field is unitary and orthogonal to $\mathcal{H}_0$. We shall usually denote the set of vector fields $\{X_1,Y_1,\ldots,X_n,Y_n\}$ by $\{Z_1,\ldots,Z_{2n}\}$. The coordinates of $\rr^{2n+1}$ will be denoted by $(x_1,y_1,\ldots,x_n,y_n,t)$, and the first $2n$ coordinates will be abbreviated by $z$ or $(x,y)$. We shall consider the map $F:\rr^{2n}\to\rr^{2n}$ defined by
\[
F(x_1,y_1,\ldots,x_n,y_n):=(-y_1,x_1,\ldots,-y_n,x_n).
\]

Given a $C^2$ function $u:\Om\subset\rr^{2n}\to\rr$ defined on an open subset $\Om$, we define the graph $G_u:=\{(z,t): z\in\Om, t=u(z)\}$. By \eqref{eq:area}, the sub-Riemannian area of the graph is given by
\[
A(G_u)=\int_{G_u} |N_h|\,dG_u,
\]
where $dG_u$ is the Riemannian metric of the graph and $|N_h|$ is the modulus of the horizontal projection of a unit normal to $G_u$. We consider on $\Om$ the basis of vector~fields $
\big\{\tfrac{\ptl}{\ptl x_1},\tfrac{\ptl}{\ptl y_1},\ldots,\tfrac{\ptl}{\ptl x_n},\tfrac{\ptl}{\ptl y_n}\big\}$.

By the Riemannian area formula
\begin{equation}
\label{eq:dgu}
dG_u=\text{Jac}\,d\mathcal{L}^{2n},
\end{equation}
where $d\mathcal{L}^{2n}$ is Lebesgue measure in $\rr^{2n}$ and $\text{Jac}$ is the Jacobian of the canonical map $\Om\to G_u$ given by
\begin{equation}
\label{eq:jac}
\text{Jac}=\det(g_{ij}+(\nabla u+F)_i(\nabla u+F)_j)^{1/2}_{i,j=1,\ldots,2n},
\end{equation}
where $g_{ij}:=g(Z_i,Z_j)$, $\nabla$ is the Euclidean gradient of $\rr^{2n}$ and $(\nabla u+F)_i$ is the $i$-th Euclidean coordinate of the vector field $\nabla u+F$ in $\Om$. We have
\[
(\nabla u+F)_i=
\begin{cases}
u_{x_{(i+1)/2}}-y_{(i+1)/2}, & i\ \text{odd}, \\
u_{y_{i/2}}+x_{i/2}, & i\ \text{even}.
\end{cases}
\]
Let us compute the composition of $|N_h|$ with the map $\Om\to G_u$. The tangent space $T G_u$ is spanned by
\begin{equation}
\label{eq:zi}
Z_i+(\nabla u+F)_i\,T, \quad i=1,\ldots,2n.
\end{equation}
So the projection to $\Om$ of the singular set $(G_u)_0$ is the set $\Om_0\subset\Om$ defined by $\Om_0:=\{z\in\Om: (\nabla u+F)(z)=0\}$. Let us compute a \emph{downward pointing} normal vector $\tilde{N}$ to $G_u$ writing
\begin{equation}
\label{eq:tilden}
\tilde{N}=\sum_{i=1}^{2n}(a_i Z_i)-T. 
\end{equation}
The horizontal component of $\tilde{N}$ is $\tilde{N}_h=\sum_{i=1}^{2n} a_iZ_i$.  We have
\[
\sum_{i=1}^{2n} a_ig_{ij}=g(\tilde{N}_h,Z_j)=g(\tilde{N},Z_j)=-(\nabla u+F)_j\escpr{\tilde{N},T}=(\nabla u +F)_j,
\]
since $Z_j$ is horizontal, $\tilde{N}$ is orthogonal to $Z_j$ defined by \eqref{eq:zi}, and \eqref{eq:tilden}. Hence 
\[
(a_1,\ldots,a_{2n})=b(\nabla u+F),
\]
where $b$ is the inverse of the matrix $\{g_{ij}\}_{i,j=1,\ldots,2n}$. So we get
\begin{equation}
\label{eq:modtilden}
|\tilde{N}|=(1+\escpr{\nabla u +F,b(\nabla u +F)})^{1/2},
\end{equation}
and
\[
|\tilde{N}_h|=\escpr{\nabla u+F, b(\nabla u+F)}^{1/2},
\]
where $\escpr{,}$ is the Euclidean Riemannian metric in $\rr^{2n}$, and so
\begin{equation}
\label{eq:modnh}
|N_h|=\frac{|\tilde{N}_h|}{|\tilde{N}|}=\frac{\escpr{\nabla u+F, b(\nabla u+F)}^{1/2}}{\big(1+\escpr{\nabla u +F,b(\nabla u +F)}\big)^{1/2}}.
\end{equation}
Observe that, from \eqref{eq:tilden} and \eqref{eq:modtilden} we also get that the scalar product of the unit normal $N$ with the Reeb vector field $T$ is given by
\begin{equation}
\label{eq:nxt}
g(N,T)=-\frac{1}{\big(1+\escpr{\nabla u+F,b(\nabla  u+F)}\big)^{1/2}}.
\end{equation}
Hence we obtain from \eqref{eq:area}, \eqref{eq:dgu}, \eqref{eq:jac} and \eqref{eq:modnh}
\begin{equation}
\label{eq:agu}
A(G_u)=\int_\Om \escpr{\nabla u+F, b(\nabla u+F)}^{1/2}\,\frac{\det(g_{ij}+(\nabla u+F)_i(\nabla u+F)_j)^{1/2}}{\big(1+\escpr{\nabla u +F,b(\nabla u +F)}\big)^{1/2}}
\,d\mathcal{L}^{2n}.
\end{equation}

Now we use formula \eqref{eq:agu} to compute the mean curvature of a graph.

\begin{lemma}
\label{1variation:graph} 
Let us consider the contact sub-Riemannian manifold $(\mathbb{R}^{2n+1},g_{\mathcal{H}_0},\omega_0)$, where $\omega_0$ is the standard contact form in $\mathbb{R}^{2n+1}$ and $g_{\mathcal{H}_0}$ is a positive definite metric in  the horizontal distribution $\mathcal{H}_0$. Let $u:\Omega\subset \rr^{2n}\rightarrow \rr$ be a ${C}^2$ function.
We denote by $g=(g_{ij})_{i,j=1,\dots,2n}$ the metric matrix and by $b=g^{-1}=(g^{ij})_{i,j=1,\dots,2n}$ the inverse metric matrix. Then the mean curvature of the graph $G_u$, computed with respect to the downward pointing normal, is given by
\begin{equation}
\label{eq:hgu}
-\divv\bigg(\frac{b(\nabla u+F)}{\escpr{\nabla u+F,b(\nabla u+F)}^{1/2}}\bigg)+\mu,
\end{equation}
where $\mu$ is a bounded function in $\Om\setminus\Om_0$, and $\divv$ is the usual Euclidean divergence in $\Om$.
\end{lemma} 

\begin{proof}
Given a smooth function $v$ with compact support in $\Om$, we shall compute the first derivative of the function $s\mapsto A(G_{u+sv})$ and we shall compare it with the general first variation of the sub-Riemannian area \eqref{eq:1stvar}. Let us fix some compact set $K\subset\Om$.

We use the usual notation in Calculus of Variations. Let us denote by
\begin{equation}
\label{eq:defG}
G(z,u,p):=\frac{\det(g_{ij}+(p+F)_i(p+F)_j)^{1/2}_{i,j=1,\ldots,2n}}{\big(1+\escpr{p +F,b(p +F)}\big)^{1/2}},
\end{equation}
where $p\in\rr^{2n}$. Observe that $G$ is a $C^\infty$ function well defined in $\Om$. From \eqref{eq:jac} and \eqref{eq:nxt}we obtain
\begin{equation}
\label{eq:simplifiedg}
G(z,u,\nabla u):=-\text{Jac}\,g(T,N).
\end{equation}
Recall that $g_{ij}=g_{ij}(z,u)$, $F=F(z)$. Let us denote also
\begin{equation}
\label{eq:defF}
F(z,u,p):=\escpr{p+F, b(p+F)}^{1/2}\,G(z,u,p).
\end{equation}
Then we can write
\[
A(G_u):=\int_\Om F(z,u,\nabla u)\,d\mathcal{L}^{2n}.
\]
So we have
\[
\frac{d}{ds}\bigg|_{s=0} A(G_{u+sv})=\int_\Om (F_u\,v+\escpr{F_p,\nabla v}\,d\mathcal{L}^{2n},
\]
where $\escpr{F_p,X}(z,u,p)=\tfrac{d}{ds}\big|_{s=0} (z,u,p+sX)$ is the gradient of $p\mapsto F(z,u,p)$. Applying the Divergence Theorem
\begin{equation}
\label{eq:1stu+sv}
\frac{d}{ds}\bigg|_{s=0} A(G_{u+sv})=\int_\Om v\,(F_u-\divv F_p)\,d\mathcal{L}^{2n}.
\end{equation}
Observe that, from \eqref{eq:defF}
\[
F_u=\frac{\escpr{p+F, \tfrac{\ptl b}{\ptl u}(p+F)}}{2\,\escpr{p+F, b(p+F)}^{1/2}}\,G+\escpr{p+F, b(p+F)}^{1/2}\,G_u,
\]
which is bounded from above since $b$ is a symmetric positive definite matrix, and so there is $C>0$ depending on $K$ so that $\escpr{\nabla u+F,b(\nabla u+F)}\ge C\,|\nabla u+F|^2$, and the numerator satisfies $\escpr{\nabla u+F,\tfrac{\ptl b}{\ptl t}(\nabla u+F)}\le C'|\nabla u+F|^2$. On the other hand
\[
F_p=G\,\frac{b(p+F)}{\escpr{p+F, b(p+F)}^{1/2}}+\escpr{p+F, b(p+F)}^{1/2}\,G_p,
\]
so that
\begin{align*}
\divv F_p=G\,\divv\bigg(\frac{b(p+F)}{\escpr{p+F, b(p+F)}^{1/2}}\bigg)&+\big<\nabla G, \frac{b(p+F)}{\escpr{p+F, b(p+F)}^{1/2}}\big>
\\
&+\divv\big(\escpr{p+F, b(p+F)}^{1/2}\,G_p\big).
\end{align*}
Observe that the last two terms are bounded and that $G_p$ is bounded, so that we get from \eqref{eq:1stu+sv} and the previous discussion
\[
\frac{d}{ds}\bigg|_{s=0} A(G_{u+sv})=\int_\Om v\,\bigg\{G\,\divv\bigg(\frac{b(\nabla u+F)}{\escpr{\nabla u+F,b(\nabla u+F)}^{1/2}}\bigg)+\mu'\bigg\}\,d\mathcal{L}^{2n},
\]
where $G$ and $\mu'$ are bounded functions in $K$.

Taking into account that the variation $s\mapsto u+sv$ is the one obtained by moving the graph $G_u$ by the one-parameter group of diffeomorphisms associated to the vector field $U:=vT$, which has normal component $g(U,N)=v\,g(T,N)$, that $dG_u=\text{Jac}\,d\mathcal{L}^{2n}$, and equation \eqref{eq:simplifiedg}, we conclude
\[
\frac{d}{ds}\bigg|_{s=0} A(G_{u+sv})=\int_\Om g(U,N)\bigg\{-\,\divv\bigg(\frac{b(\nabla u+F)}{\escpr{\nabla u+F,b(\nabla u+F)}^{1/2}}\bigg)+\mu\bigg\}\,dG_u,
\]
where $\mu:=\mu' (g(N,T)\,\text{Jac})^{-1}$ is a bounded function. Comparing this formula with the general first variation one \eqref{eq:1stvar}, and taking into account that $g(U,N)$ is arbitrary we get \eqref{eq:hgu}.
\end{proof}

\begin{remark}
 If $g=g_0$ is the standard Riemannian metric in the Heisenberg group so that $\{X_1,Y_1,\ldots,X_n,Y_n,T\}$ is orthonormal then $(g_{ij})_{i,j=1,\ldots,2n}$ is the identity matrix, $b=\text{Id}$, $\mu=0$,  and we have the usual mean curvature equation, see \cite{CHY}.
\end{remark} 
 
\begin{lemma} 
\label{lem:foliation}
Let $(M,g_\mathcal{H},\omega)$ be a contact sub-Riemannian manifold. Given $p\in M$, there exists a neighborhood $U$ of $p$ so that $U\setminus\{p\}$ is foliated by surfaces with mean curvature uniformly bounded outside any neighborhood $V\subset U$ of $p$.
\end{lemma}

\begin{proof}
Since the result is local, we may assume, using a Darboux's chart, that our contact sub-Riemannian manifold is $(\rr^{2n+1},g,\omega_0)$, where $\omega_0$ is the standard contact form in \eqref{eq:omega0} and $g$ is an arbitrary positive definite metric in the horizontal distribution $\mathcal{H}_0$. We also assume $p=0$.

For each $\la>0$, we consider the hypersurface $\sph_\la$ given by the graph of the function
\begin{equation}
\label{eq:ulambda}
u_\la(z)=\frac{1}{2\lambda^2}\{ \lambda |z|(1-\lambda^2|z|^2)^{1/2}+\arccos (\lambda |z|)  \}, \quad |z|\leq \frac{1}{\lambda},
\end{equation}
and its reflection with respect to the hyperplane $t=0$, see \cite{Ri2}. Each $\sph_\la$ is a topological sphere of class $C^2$ with constant mean curvature $\la$ in the Heisenberg group $\hh^{2n+1}$ and two singular points $\pm(0,\pi/(4\la^2))$. The family $\{\sph_{\lambda}\}_{\la>0}$ is a foliation of $\rr^{2n+1}\setminus\{0\}$. From now on we fix some $\la>0$ and let $u:=u_\la$.

From Lemma~\ref{1variation:graph} it is sufficient show that
\begin{equation}\label{eq:meancur}
\divv\left(\frac{b(\nabla u+F)}{\escpr{\nabla u+F,b(\nabla u+F)}^{1/2}}\right)
\end{equation}
is bounded near the singular points, In fact the mean curvature is continuous away from the singular set by the regularity of $\sph_\lambda$.

Let $g^i:=(g^{i1},\ldots, g^{i(2n)})$ be the vector in $\rr^{2n}$ corresponding to the $i$-th row of the matrix $b$. We have 
\[
\divv\left(\frac{b(\nabla u+F)}{\escpr{\nabla u+F,b(\nabla u+F)}^{1/2}}\right)=\sum_{i=1}^{2n}\ptl_i\left(\frac{\escpr{g^i,\nabla u +F}}{\escpr{\nabla u+F,b(\nabla u+F)}^{1/2}}\right),
\]
where $\partial_i$ is the partial derivative with respect the $i$-th variable, i.e., $x_{(i+1)/2}$ when $i$ is odd and $y_{i/2}$ when $i$ is even. Taking derivatives we get
\begin{multline*}
\sum_{i=1}^{2n}\ptl_i\left(\frac{\escpr{g^i,\nabla u +F}}{\escpr{\nabla u+F,b(\nabla u+F)}^{1/2}}\right)=
\sum_{i=1}^{2n}\frac{\escpr{\ptl_ig^i,\nabla u+F}+\escpr{g^i,\ptl_i(\nabla u+F)}}{\escpr{\nabla u+F,b(\nabla u+F)}^{1/2}}
\\
-\escpr{g^i,\nabla u +F}\,\frac{\tfrac{1}{2}\,\escpr{\nabla u+F,(\ptl_ib)(\nabla u+F)}+\escpr{\ptl_i(\nabla u+F),b(\nabla u+F)}}{\escpr{\nabla u+F,b(\nabla u+F)}^{3/2}}.
\end{multline*}
It is clear that the first and the third summands are bounded. So we only have to prove that
\begin{multline}
\label{eq:step1}
\sum_{i=1}^{2n}\escpr{g^i,\ptl_i(\nabla u+F)}\escpr{\nabla u+F,b(\nabla u+F)}
\\
-\sum_{i,j,k,\ell=1}^{2n}\escpr{g^i,\nabla u+F}\escpr{\ptl_i(\nabla u+F),b(\nabla u+F)}\le C\,|\nabla u+F|^{3/2}
\end{multline}
for some positive constant $C$. We easily see that the left side of \eqref{eq:step1} is equal to
\begin{multline}
\label{eq:step2}
\sum_{i,j,k,\ell=1}^{2n}g^{ij}g^{k\ell}\ptl_i(\nabla u+F)_j(\nabla u+F)_k(\nabla u+F)_\ell
\\
-\sum_{i,j,k,\ell=1}^{2n} g^{ij}g^{k\ell}\ptl_i(\nabla u+F)_k (\nabla u+F)_j(\nabla u+F)_\ell.
\end{multline}
Taking into account that
\[
\ptl_i F_j+\ptl_jF_i=0, \quad \text{for all}\ i,j=1,\ldots,2n,
\]
and the symmetries of \eqref{eq:step2}, we get that \eqref{eq:step2} is equal to 
\[
\sum_{i,j,k,\ell=1}^{2n}g^{ij}g^{k\ell}u_{ij}(\nabla u+F)_k(\nabla u+F)_\ell
-\sum_{i,j,k,\ell=1}^{2n} g^{ij}g^{k\ell}u_{ik}(\nabla u+F)_j(\nabla u+F)_\ell,
\]
so we only need to show that each term
\begin{equation*}
\frac{ (\nabla u+F)_k (\nabla u+F)_\ell\, u_{ij}}{\langle \nabla u+F,b(\nabla u+F)\rangle^{3/2}}
\end{equation*}
is bounded to complete the proof. Since
\[
|(\nabla u+F)_k|\le |\nabla u +F|, \ \text{and }\escpr{\nabla u+F,b(\nabla u+F}^{1/2}\ge C |\nabla u+F|,
\] 
for some positive constant $C>0$, it is enough to show that
\begin{equation}
\label{eq:uij}
\frac{u_{ij}}{|\nabla u+F|}
\end{equation}
is bounded.

A direct computation yields
\[
\frac{\ptl u}{\ptl x_i}=-\frac{\la |z|x_i}{(1-\la^2|z|^2)^{1/2}},
\qquad
\frac{\ptl u}{\ptl y_i}=-\frac{\la |z|y_i}{(1-\la^2|z|^2)^{1/2}},
\]
and so
\[
|\nabla u+F|^2=\sum_{i=1}^n \bigg(\frac{\ptl u}{\ptl x_i}-y_i\bigg)^2+\bigg(\frac{\ptl u}{\ptl y_i}+x_i\bigg)^2=|z|^2\bigg(1+\frac{\la^2|z|^2}{1-\la^2|z|^2}\bigg).
\]
Hence 
\begin{equation}
\label{eq:nablau+f}
C_1|z|\le|\nabla u+F|\le C_2|z|,
\end{equation}
for some constants $C_1$, $C_2>0$ near $z=0$.

On the other hand
\begin{align*}
\frac{\ptl^2u}{\ptl x_i\ptl x_j}&=-\delta_{ij}\,\frac{\lambda |z|}{(1-\la^2|z|^2)^{1/2}}-\frac{\lambda x_ix_j}{|z|(1-\la^2|z|^2)^{3/2}},
\\
\frac{\ptl^2u}{\ptl y_i\ptl y_j}&=-\delta_{ij}\,\frac{\lambda |z|}{(1-\la^2|z|^2)^{1/2}}-\frac{\lambda y_iy_j}{|z|(1-\la^2|z|^2)^{3/2}},
\\
\frac{\ptl^2u}{\ptl x_iy_j}&=-\frac{\lambda x_iy_j}{|z|(1-\la^2|z|^2)^{3/2}},
\end{align*}
and so
\[
|u_{ij}|\le C|z|,
\]
for some constant $C>0$. This inequality, together with \eqref{eq:nablau+f}, shows that \eqref{eq:uij} is bounded.

%
\end{proof}

\begin{lemma}[Deformation Lemma]
\label{addingperimeter} 
Let $(M,g_\mathcal{H},\omega)$ be a contact sub-Riemannian~manifold and $\Om\subset M$ a finite perimeter set. Then there exists a small deformation $\widetilde{\Om}_r\supset\Om$, $0<r\le r_0$, such that
\begin{equation*}
\per(\partial(\widetilde{\Om}_r-\Om))\leq C\,|\widetilde{\Om}_r-\Om|,
\end{equation*}
where $C$ is a positive constant.
\end{lemma}

\begin{proof}
For $p\in\intt(\Om)$ sufficiently close to $\partial\Om$, there exists by Lemma~\ref{lem:foliation} a local foliation by hypersurfaces $F_r$, $0<r\le r_0$, with mean curvature uniformly bounded outside a small neighborhood of $p$. Let $U_r$ the regions bounded by $F_r$  and let $\nu_h(q)$ the horizontal unit normal at $q\in F_r$ of the surface $F_r$, for $r\in[d(p,\partial \Omega), r_0]$.
Letting $\Om_r:= \Om^c\cap U_r$, $\widetilde{\Om}_r:=\Om_r\cup\Om$, we have that there exists $C>0$ so that $\divv\nu_h\le C$ by \eqref{eq:hdivsgfinal} and the boundedness of the mean curvature. So we have
\begin{equation*}
\begin{aligned}
C |\Om_r|&\geq \int\limits_{\Om_r} \divv (\nu_h) dv_g=\int\limits_{M} g_{\mathcal{H}}(\nu_h,\nu)dP(F_r\cap \Omega^c)+\int\limits_{M} g_{\mathcal{H}}(\nu_h,\nu) dP(\partial^*\Om_r\cap U_r)\\
&\geq \per(\partial F_r\cap \Om^c)-\per(\partial^*\Om_r\cap U_r),
\end{aligned}
\end{equation*}
where $\nu$ is defined in the Gauss-Green formula. We have used $g_{\mathcal{H}}(\nu_h,\nu)\equiv 1$ in the first integral and $g_{\mathcal{H}}(\nu_h,\nu)\geq -1$ in the second one. But also that from the definition of $dP(.)$ it follows  
\begin{equation*}
\int\limits_{\Omega}dP(E)=P(E,\Omega),
\end{equation*}
see \cite[p.~879--880]{FSSC5} and \cite[p.~491--494]{FSSC4}.
\end{proof}

\begin{lemma}
\label{lem:bounded}
Let $(M,g_\mathcal{H},\omega)$ be a contact sub-Riemannian manifold, and  $E\subset M$ be a set minimizing perimeter under a volume constraint. Then $E$ is bounded. 
\end{lemma}

\begin{proof}
We fix $p\in M$ and denote the ball $B(p,r)$ by $B_r$. We let $V(r):=|E\cap (M\setminus B_r)|$, so that $V(r)\to 0$ when $r\to\infty$ since $E$ has finite volume. Let us assume that $V(r)>0$ for all $r>0$. Applying the isoperimetric inequality for small volumes when $r$ is large enough to the set $E\cap (M\setminus B_r)$ we get, taking $q$ as in \eqref{eq:defq},
\begin{align}
\label{eq:isopV(r)}
C_I\,V(r)^q&\le\per(E\cap (M\setminus B_r))
\\ \notag
&\le\per(E,M\setminus\overline{B}_r)+\per(E\cap B_r,\ptl B_r)
\\ \notag
&\le\per(E,M\setminus\overline{B}_r)+|V'(r)|
\\ \notag
&\le\per(E)-\per(E,B_r)+|V'(r)|.
\end{align}

We now fix some $r_0>0$. For $r>r_0$, the Deformation Lemma shows the existence of a set ${E}_r$ so that $E_r$ is a small deformation of $E\cap B_r$, ${E}_r\setminus (E\cap B_r)$ is properly contained in $B_{r_0}$, $|E_r|=|E|$ (which implies $|E\setminus E_r|=V(r)$), and $\per(E_r,B_r)\le\per(E,B_r)+C\,V(r)$. So we have
\begin{align}
\label{eq:perEr}
\per(E_r)&\le\per(E_r,B_r)+\per(E_r\cap B_r,\ptl B_r)
\\ \notag
&=\per(E_r,B_r)+\per(E\cap B_r,\ptl B_r)
\\ \notag
&\le\per(E_r,B_r)+|V'(r)|
\end{align}
By the isoperimetric property of $E$ we also have
\begin{equation}
\label{eq:isopE}
\per(E)\le\per(E_r)\le\per(E_r,B_r)+|V'(r)|,
\end{equation}
for all $r\ge r_0$.

From \eqref{eq:isopV(r)}, \eqref{eq:perEr} and \eqref{eq:isopE} we finally get
\begin{equation}
\label{eq:ineqvr}
C_IV(r)^q\le C\,V(r)+2\,|V'(r)|.
\end{equation}
Since $V(r)=V(r)^{1-q}V(r)^q\le (C_I/2)\,V(r)^q$ for $r$ large enough, we get
\[
-\frac{C_I}{2}\,V(r)^q\ge 2\,V'(r),
\]
or, equivalently,
\[
(V^{1/Q})'\le -\frac{C_I Q}{2}<0,
\]
which forces $V(r)$ to be negative for $r$ large enough. This contradiction proves the result.
\end{proof}

\section{Structure of minimizing sequences}
\label{sec:structure}

In this section we will prove an structure result for minimizing sequences in a non-compact contact sub-Riemannian manifold. Partial versions of this result were obtained for Riemannian surfaces, \cite{MR1883725}, \cite{MR1857855}, and for Riemannian manifolds \cite{Ri-Ro2}.

\begin{proposition}
\label{prop:minseq}
Let $(M,g_\mathcal{H},\omega)$ be a non-compact contact sub-Riemannian manifold. Consider  a minimizing sequence $\{E_k\}_{k\in\mathbb{N}}$ of sets of volume $v$ converging in $L^1_{loc}(M)$ to a finite perimeter set $E\subset M$, eventually empty. Then there exist sequen\-ces of finite perimeter sets $\{E^c_k\}_{k\in\mathbb{N}}$, $\{E^d_k\}_{k\in\mathbb{N}}$ such that
\begin{enumerate}
\item\label{one} $\{E_k^c\}_{k\in\nn}$ converges to $E$ in $L^1(M)$, $\{E_k^d\}_{k\in\nn}$ diverges, and $|E_k^c|+|E^d_k|=v$. 
\item\label{two} $\lim_{k\rightarrow\infty} \per(E_k^c)+\per(E_k^d)=I_M(v)$.
\item\label{three} $\lim_{k\to\infty}\per(E^c_k)=\per(E)$. 
\item\label{four} If $|E|\neq 0$, then $E$ is an isoperimetric region of volume $|E|$.
\item\label{five} Moreover, if $M/\Isom_\omega(M,g)$ is compact then $\lim_{k\to\infty}P(E_k^d)=I_M(v-|E|)$. In particular, $I_M(v)=I_M(|E|)+I_M(v-|E|)$.
\end{enumerate}
\end{proposition}

\begin{proof}
We fix a point $p\in M$ and we consider the balls $B(r):=B(p,r)$. Let $m(r):=|E\cap B(p,r)|$, $m_k(r):=|E_k\cap B(r)|$.

We can choose a sequence of diverging radii $r_k>0$ so that, considering a subsequence of $\{E_k\}_{k\in\nn}$, we would had
\begin{gather}
\label{eq:conv1}\int\limits_{B(r_k)}|\mathbf{1}_E-\mathbf{1}_{E_k}|\le\frac{1}{k},
\\
\label{eq:conv2}P(E_k\setminus B(r_k),\ptl B(r_k))\le\frac{v}{k}.
\end{gather}
In order to prove \eqref{eq:conv1} and \eqref{eq:conv2} we consider a sequence of radii $\{s_k\}_{k\in\nn}$ so that $s_{k+1}-s_{k}\ge k$ for all $k\in\nn$. Taking a subsequence of $\{E_k\}_{k\in\nn}$, we may assume that
\[
\int\limits_{B(s_{k+1})}|\mathbf{1}_E-\mathbf{1}_{E_k}|\le\frac{1}{k},
\]
so that \eqref{eq:conv1} holds for all $r\in (0,s_{k+1})$.  To prove \eqref{eq:conv2} we observe that $m_k(r)$ is an increasing function. By Lebesgue's Theorem
\[
\int\limits_{s_{k}}^{s_{k+1}} m'(r)\,dr\le m(s_{k+1})-m(s_k)\le v,
\]
which implies that there is a set of positive measure in $[s_k,s_{k+1}]$ so that $m'(r)\le\tfrac{v}{k}$. By Ambrosio's localization Lemma \cite[Lemma~3.5]{Am1} we have, for almost everywhere $r$,
\[
P(E_k\setminus B(r)),\ptl B(r))\le m_k'(r).
\]
This implies that there is $r_k\in [s_k,s_{k+1}]$ so that \eqref{eq:conv2} holds.

Now we define
\[
E_k^c:=E\cap B(r_k), \qquad E_k^d:=E\setminus B(r_{k+1}).
\]

Now we prove \ref{one}. Since $E$ has finite volume and \eqref{eq:conv1} holds we conclude that $\{E_k^c\}_{k\in\nn}$ converges in $L^1(M)$ to $E$. The divergence of the sequence $\{E_k^d\}_{k\in\nn}$ and equality $|E_k^c|+|E_k^d|=v$ follow from the definitions of $E_k^c$ and $E_k^d$.

In order to prove \ref{two} we take into account that
\begin{align*}
P(E_k^c)&\le P(E_k,B(r_k))+P(E_k\cap\ptl B(r_k),\ptl B(r_k)),
\\
P(E_k^d)&\le P(E_k,M\setminus \overline{B}(r_k))+P(E_k\cap\ptl B(r_k),\ptl B(r_k)).
\end{align*}
By \eqref{eq:conv2} we have
\[
P(E_k)\le P(E_k^c)+P(E_k^d)\le P(E_k)+\frac{2v}{k}.
\]
Taking limits when $k\to\infty$ we get \ref{two}.

To prove \ref{three} we shall first show that
\begin{equation}
\label{eq:proof3}
P(E)=\liminf_{k\to\infty} P(E_k^c)
\end{equation}
reasoning by contradiction. Since $E_k^c$ converges in $L^1(M)$ to $E$, we may assume that the strict inequality $P(E)<\liminf_{k\to\infty} P(E_k)$ holds. Reasoning as above we obtain an non-decreasing and diverging sequence of radii $\{\rho_k\}_{k\in\nn}$ so that $\rho_k<r_k$ and
\[
P(E\cap\ptl B(\rho_k),\ptl B(\rho_k))\le\frac{v}{k},
\]
for all $k\in\nn$. Let $E_k':=E\cap B(\rho_k)$. The perimeter of $E_k'$ satisfies
\[
P(E_k')\le P(E,B(\rho_k))+P(E\cap\ptl B(\rho_k),\ptl B(\rho_k))\le P(E)+\frac{v}{k},
\]
and for the volume $|E_k'|$ we have
\[
\lim_{k\to\infty} |E'_k|=|E|=v-\lim_{k\to\infty} |E_k^d|.
\]
We fix two points $p_1\in\intt(E)$, $p_2\in\intt(M\setminus E)$, close enough to the boundary of $E$, so that we can apply the Deformation Lemma in a neighborhood of each point. This allows us to make small corrections of the volume and obtain, for $k\in\nn$ large enough, a set $E_k''$ of finite perimeter so that
\[
|E_k''|+|E_k^d|=v,
\]
and
\[
P(E_k'')\le P(E_k')+C\,\big||E_k'|-|E_k^d|\big|\le P(E)+\frac{v}{k}+C\,\big||E_k'|-|E_k^d|,
\]
so that
\[
\liminf_{k\to\infty} P(E_k'')\le P(E).
\]
Then $F_k:=E_k''\cup E_k^d$ is sequence of sets of volume $v$ with
\[
\liminf_{k\to\infty} P(F_k)\le P(E)+\liminf_{k\to\infty} P(E_k^d)<\liminf_{k\to\infty} (P(E_k^c)+P(E_k^d))=I_M(v),
\]
which clearly gives us a contradiction and proves \eqref{eq:proof3}. To complete the proof of \ref{three} we observe that we can replace the inferior limit in \eqref{eq:proof3} by the true limit of the sequence since every subsequence of a minimizing sequence is also minimizing.

To prove \ref{four} we consider a finite perimeter set $F$ with $|F|=|E|$ and $P(F)<P(E)$ and we reason as in the proof of \ref{three} with $F$ instead of $E$.

Let us finally see that \ref{five} holds. From \ref{two} and \ref{three} we see that $\lim_{k\to\infty}P(E_k^d)$ exists and it is equal to $I_M(v)-P(E)$. If this limit were smaller than $I_M(v-|E|)$ then we could slightly modify the sequence $\{E_k^d\}_{k\in\nn}$ to produce another one $\{F_k\}_{k\in\nn}$ with $|F_k|=v-|E|$ and $\lim_{k\to\infty} P(F_k)=\lim_{k\to\infty} P(E_k^d)<I_M(v-|E|)$, which gives a contradiction. If $\lim_{k\to\infty}P(E_k^d)$  were larger than $I_M(v-|E|)$ then we could find a set $F$ with $|F|=v-|E|$ so that
\[
I_M(v-|E|)<P(F)<\lim_{k\to\infty}P(E_k^d).
\]
Modifying again slightly the volume of $F$ we produce a sequence $\{F_k\}_{k\in\nn}$ so that $|E|+|F_k|=v$ and $\lim_{k\to\infty} P(F_k)=P(F)$. Since $E$ is bounded, we can translate the sets $F_k$ so that they are at positive distance from $E$. Hence
\[
\lim_{k\to\infty}P(E\cup F_k)=\lim_{k\to\infty} P(E)+P(F_k)=P(E)+P(F)<I_M(v),
\]
a contradiction that proves \ref{five}.
\end{proof}

\begin{remark}
The proof of the first three items in the statement of Proposition~\ref{prop:minseq} works in quite general metric measure spaces. The proof of the last two ones needs the compactness of the isoperimetric regions.
\end{remark}

\section{Proof of the main result}
\label{sec:main}

We shall prove in this section our main result

\begin{theorem}
\label{th:main}
Let $(M,g_\mathcal{H},\omega)$ be a contact sub-Riemannian manifold such that the quotient $M/\Isom_\omega(M,g)$ is compact. Then, for any $0<v<|M|$, there exists on $M$ an isoperimetric region of volume $v$.
\end{theorem}

First we need the following result \cite[Lemma~4.1]{Le-Ri}

\begin{lemma}
\label{lem:lr}
Let $E\subset M$ be a set with positive and finite perimeter and measure. Assume that $m\in (0,\inf_{x\in M} |B(x,r_0)|/2)$, where $r_0>0$ is the radius for which the relative isoperimetric inequality holds, is such that $|E\cap B(x,r_0)|<m$ for all $x\in M$. Then we have
\begin{equation}
\label{eq:lr}
C\,|E|^Q\le m\,P(E)^Q,
\end{equation}
for some constant $C>0$ that only depends on $Q$.
\end{lemma}

\begin{proof}
We closely follow the proof of \cite[Lemma~4.1]{Le-Ri}. We consider a maximal family of points $\mathcal{A}$ in $M$ so that $d(x,x')\ge r_0/2$ for all $x$, $x'\in\mathcal{A}$, $x\neq x'$, and $|E\cap B(x,r_0/2)|>0$ for all $x\in\mathcal{A}$. Then $\bigcup_{x\in\mathcal{A}} B(x,r_0)$ cover almost all of $E$. We have
\begin{align*}
|E|&\le\sum_{x\in\mathcal{A}} |E\cap B(x,r_0)|\le m^{1/Q}\sum_{x\in\mathcal{A}} |E\cap B(x,r_0)|^q
\\
&\le m^{1/Q}C_I\sum_{x\in\mathcal{A}} P(E,B(x,r_0)),
\end{align*}
since $(1/Q)+q=1$ and $|E\cap B(x,r_0)|<m$. The last inequality follows from the relative isoperimetric inequality since $|E\cap B(x,r_0)|<m\le |B(x,r_0)|/2$ and so $\min\{|E\cap B(x,r_0)|,|E^c\cap B(x,r_0)|\}=|E\cap B(x,r_0)|$. The overlapping is controlled in the same way as in \cite{Le-Ri} to conclude the proof.
\end{proof}

Using the following result we can prove Proposition~\ref{prop:cv0}
\begin{lemma}[{\cite[Thm.~4.3]{Am}}]
\label{estimateperimeterinballs} 
The measure $P(E,\cdot)$ satisfies
\begin{equation*}
\tau< \liminf\limits_{\delta\rightarrow 0} \frac{P(E,B(x,\delta))}{\delta^{Q-1}}\leq \limsup\limits_{\delta\rightarrow 0} \frac{P(E,B(x,\delta))}{\delta^{Q-1}} <+\infty,
\end{equation*}
for $P(E,\cdot)$-a.e. $x\in M$, with $\tau>0$.
\end{lemma}

\begin{proposition}
\label{prop:cv0}
Given $v_0>0$, there exists a constant $C(v_0)>0$ so that
\begin{equation}
\label{eq:upperI}
I_M(v)\leq C(v_0)\,v^{(Q-1)/Q},
\end{equation}
for all $v\in (0,v_0]$.
\end{proposition}

\begin{proof} For any $x\in M$ we have
\begin{equation*}
I_M(|B(x,r)|)\leq P(B(x,r))\leq c r^{Q-1}\leq \frac{c}{C^{(Q-1)/Q}}|B(x,r)|^{(Q-1)/Q},
\end{equation*}
where we have used $|B(x,r)|\geq C r^Q$ to get $r^Q\leq C^{-1/Q}|B(x,r)|^{1/Q}$ and Lemma~\ref{estimateperimeterinballs} with $E=B(x,r)$ and $\delta=2r$.

\end{proof}

\begin{proof}[Proof of Theorem~\ref{th:main}]
We fix a volume $0<v<|M|$, and we consider a minimizing sequence $\{E_k\}_{k\in\nn}$ of sets of volume $v$ whose perimeters approach $I_M(v)$. In case $M$ is compact, we can extract a convergent subsequence to a finite perimeter set $E$ with $|E|=v$ and $\per(E)=I_M(v)$. 

We assume from now on that $M$ is not compact. By Lemma~\ref{lem:lr}, for any $m>0$ such that $mv<\inf_{x\in M} |B(x,r_0)|/2$, there is a a constant $C>0$, only depending on $Q$, so that, for any finite perimeter set $E\subset M$ satisfying $|E\cap B(x,r_0)|<m\,|E|$ for all $x\in M$, we have  
\begin{equation*}
C|E|^Q\le (m|E|)\,P(E)^Q,
\end{equation*}
and so
\begin{equation}
\label{eq:conslr}
P(E)\ge \bigg(\frac{C}{m}\bigg)^{1/Q}\,|E|^{(Q-1)/Q}.
\end{equation}
From Proposition~\ref{prop:cv0} we deduce that, given $v>0$, there is a constant $C(v)>0$ so that $I_M(w)\le C(v)\,w^{(Q-1)/Q}$ for all $w\in (0,v]$. Taking $m_0>0$ small enough so that
\begin{equation}
\label{eq:m0small}
\bigg(\frac{C}{m_0}\bigg)^{1/Q}|E|^{(Q-1)/Q}>2 C(v)\,|E|^{(Q-1)/Q}
\end{equation}
we conclude, using \eqref{eq:conslr}, \eqref{eq:m0small} and \eqref{eq:upperI}
\begin{equation}
\label{eq:cont}
P(E)\ge 2\,I_M(|E|).
\end{equation}
We conclude from \eqref{eq:cont} that, for $k$ large enough, the sets in the minimizing sequence $\{E_k\}_{k\in\nn}$ cannot satisfy the property $|E\cap B(x,r_0)|<m|E|$ for all $x\in M$. So we can take  points $x_k\in M$ such that
\begin{equation*}
|E_k\cap B(x_k,r_0)|\ge m_0 |E_k|=m_0v,
\end{equation*}
for $k$ large enough. Since $M/\text{Isom}_\omega(M,g)$ is compact, we translate the whole minimizing sequence (and still denote it in the same way), so that $\{x_k\}_{k\in\nn}$ is bounded. By passing to a subsequence, denoted in the same way, we assume that $\{x_k\}_{k\in\nn}$ converges to some point $x_0\in M$. By the compactness Lemma there is a convergent subsequence, still denoted by $\{E_k\}_{k\in\nn}$ that converges to some finite perimeter set $E$, and
\[
m_0 v\le\liminf_{k\to\infty} |E_k\cap B(x_0,r_0)|=|E\cap B(x_0,r_0)|,
\]
and
\[
|E|\le\liminf_{k\to\infty}|E_k|=v.
\]
So we have proven the following fact: from every minimizing sequence of sets of volume $v>0$, one can produce, suitably applying isometries of $M$ to each member of the sequence, a new minimizing sequence $\{E_k\}_{k\in\nn}$ that converges to some finite perimeter set $E$ with $m_0v\le |E|\le v$, where $m_0>0$ is a universal constant that only depends on $v$. Hence a fraction of the total volume is captured by the minimizing sequence.

Now take a minimizing sequence $\{E_k\}_{k\in\nn}$ that converges to some finite perimeter set $E$ of volume $m_0v\le |E|<v$. The set $E$ is isoperimetric for volume $|E|$ and hence bounded by Lemma~\ref{lem:bounded}. By Proposition~\ref{prop:minseq}, the sequence $\{E_k\}_{k\in\nn}$ can be replaced by another minimizing sequence $\{E_k^c\cup E_k^d\}_{k\in\nn}$ so that $E_k^c\to E$ and $E_k^d$ diverges. Moreover, $\{E_k^d\}_{k\in\nn}$ is minimizing for volume $v-|E|$. Hence one obtains
\[
I_M(|E|)+I_M(v-|E|)=I_M(v).
\]

If $|E|=v$ we are done since $P(E)\le\liminf_{k\to\infty} P(E_k)=I_M(|E|)$ and hence $E$ is an isoperimetric region. So assume that $|E|<v$ and observe that $|E|\ge m_0\,v$. It is clear that $E$ is an isoperimetric region of volume $|E|$. The minimizing sequence can be broken into two pieces: one of them converging to $E$ and the other one diverging. The diverging part is a minimizing sequence for volume $v-|E|$. We let $F_0:=E$.

Now we apply again the previous arguments to the diverging part of the sequence, which is minimizing for volume $v-|E|$. We translate the sets to capture part of the volume and we get a new isoperimetric region $F_1$ with volume
\[
v-|F_0|\ge |F_1|\ge m_0(v-|F_0|),
\]
and a new diverging minimizing sequence for volume $v-|F_0|-|F_1|$. By induction we get a sequence of isoperimetric regions $\{F_k\}_{k\in\nn}$ so that the volume of $F_k$ satisfies
\[
|F_k|\ge m_0\bigg(v-\sum_{i=0}^{k-1}|F_i|\bigg).
\]
Hence we have
\[
\sum_{i=0}^k |F_i|\ge (k+1)m_0 v-km_0\sum_{i=0}^{k-1} |F_i|
\ge (k+1)m_0 v-km_0\sum_{i=0}^{k} |F_i|,
\]
and so
\[
\sum_{i=0}^k |F_i|\ge \frac{(k+1)m_0v}{1+km_0}.
\]
Taking limits when $k\to\infty$ we get
\[
\lim_{k\to\infty}\sum_{i=0}^k |F_i|=v.
\]
Moreover,
\[
\sum_{i=0}^\infty P(F_i)=I_M(v).
\]
Each region $F_i$ is bounded, so that we can place them in $M$ using the isometry group so that they are at positive distance (each one contained in an annulus centered at some given point). Hence $F:=\bigcup_{i=0}^\infty F_i$ is an isoperimetric region of volume $v$. In fact, $F$ must be bounded by Lemma~\ref{lem:bounded}, so we only need a finite number of steps to recover all the volume.

%
%
\end{proof}

\bibliography{pseudo-hermitian}

\end{document}